\RequirePackage{fix-cm}
\documentclass[smallextended]{svjour3}
\smartqed
\usepackage{graphicx}
\usepackage{amsmath,amssymb,amsfonts,latexsym}
\providecommand{\diff}{\mathop{}\!\mathrm{d}}

\begin{document}

\title{Viscosity solutions of the integro-differential equation for the Cram\'{e}r--Lundberg model with annuity payments and investments}

\titlerunning{Viscosity solutions for the Cram\'{e}r--Lundberg model}

\author{Platon Promyslov}


\institute{P. Promyslov \at
              HSE University, Faculty of Computer Science, Moscow, 101000, Russia \\
              \email{platon.promyslov@gmail.com}           
}

\date{Received: date / Accepted: date}

\maketitle

\begin{abstract}
This paper examines the dynamics of an insurance company's surplus within the framework of the Cram\'{e}r--Lundberg model with annuity payments and investments. It is established that the survival probability is the unique bounded viscosity solution of the corresponding singular second-order integro-differential equation. The uniqueness of the solution in the class of bounded functions is proved using a proven non-local comparison principle. The main result of the work is the justification of smoothness: by leveraging the properties of the semicontinuous envelopes of viscosity solutions, the continuity of the survival function is proved. This allows the problem to be reduced to local elliptic regularity and demonstrates that the solution belongs to the class $C^2((0,\infty))$. Thus, the correctness of the classical integro-differential equation is established without a priori probabilistic assumptions regarding smoothness.

\end{abstract}

\section{Introduction}

The collective risk theory, whose foundations were laid in the works of Lundberg and Cram\'{e}r, traditionally considered the dynamics of an insurance company's reserves without accounting for investment income (see \cite{Asmussen2010}). However, in modern economic realities, investing reserves in financial assets plays a critical role. This allows for hedging against inflationary risks but simultaneously introduces substantial changes to the model. The foundations for modeling such processes were described in the paper by Paulsen \cite{Paulsen1993}.

In the study of ruin probability for models with investments, two main analytical approaches have historically emerged. The first relies on implicit renewal theory and equations in the sense of distributions (see \cite{Goldie1991}). This method is highly effective for finding the exact asymptotics of the ruin probability as the initial capital tends to infinity (see, for example, \cite{Frolova2002,KLP2026,KPro2023}).

The second approach links the survival probability to solutions of integro-differential equations (IDEs), which allows for the global analysis of the unknown function's behavior over its entire domain. The formal derivation of the IDE using It\^{o}'s formula is a standard step; however, it strictly requires a priori twice continuous differentiability of the survival function. Justifying this smoothness presents a non-trivial analytical problem. In earlier works, it was often merely postulated (see, e.g., \cite{PG1997,WW2001}). In \cite{KP2016}, this problem was solved for the annuity model, but at the cost of employing heavy probabilistic machinery based on the subtle properties of exponential functionals of Brownian motion.

An alternative to classical analysis is the theory of viscosity solutions (see, e.g., \cite{Crandall1992,Fleming1993}), which allows for handling non-smooth functions and coefficients that degenerate at the boundary. For non-local integro-differential operators, this approach was deeply developed in \cite{Barles2008}. In \cite{BK2015}, the machinery of viscosity solutions was applied to a non-life insurance model, where the capital experiences downward jumps. In a remark, the authors of \cite{BK2015} indicated that the dual model (with annuity payments and upward jumps) could be investigated analogously.

However, a direct transfer of their arguments to the annuity model encounters serious obstacles. In the classical non-life insurance model, large claims can instantaneously drive the capital to a negative value (the process ``jumps'' across the boundary). In the annuity model, the jumps are positive, so ruin cannot occur as a result of a jump. Crossing the zero level is driven exclusively by the continuous diffusion component and the drift, which implies that at the moment of ruin $\tau$, the capital is exactly equal to zero ($X_\tau = 0$ almost surely). As a consequence, the diffusion coefficient of the equation, which is proportional to $u^2$, degenerates at zero, rendering the boundary value problem highly singular. This specific feature necessitates a fundamentally different formulation of the boundary conditions, a modification of the touching conditions for the non-local operator, and the construction of new estimates to prove the comparison principle.

The present paper not only fills the gap in applying viscosity solutions to the annuity model but also substantially strengthens the results obtainable by this method. While the investigation in \cite{BK2015} is limited to proving the existence and uniqueness of a (non-smooth) viscosity solution, the current work employs the viscosity solution as an intermediate tool to prove the existence of a classical solution.

The aim of this paper is to analytically justify the correctness of the classical IDE for the survival probability in the annuity model using methods from the theory of differential equations, avoiding cumbersome probabilistic analysis. The main results of the paper are as follows:
\begin{enumerate}
    \item It is proved that the survival probability $\Phi(u)$ is a viscosity solution of a singular second-order IDE. The definition of the solution is adapted to the specific nature of the operator with one-sided jumps.
    \item The uniqueness of the viscosity solution in the class of bounded functions is established. For this purpose, a strict comparison principle for the degenerate non-local operator is proved.
    \item The smoothness of the solution is demonstrated. Using the machinery of viscosity solutions, the continuity of the unknown function is proved, which allows invoking classical local elliptic regularity to upgrade the regularity of the viscosity solution to the class $C^2((0, \infty))$. 
\end{enumerate}

Thus, the survival probability is shown to be a classical solution of the boundary value problem without any a priori assumptions about its smoothness.

The paper is organized as follows. Section \ref{sec:model} describes the mathematical model of the capital dynamics and establishes the boundary properties. Section \ref{sec:viscosity_definition} provides the definition of a viscosity solution for the operator under consideration. Section \ref{sec:existence} is devoted to the proof of the existence of the solution, and Section \ref{sec:comparison} to the comparison principle and uniqueness. In the concluding Section \ref{sec:regularity}, the smoothness of the solution is established, and the transition from a viscosity solution to a classical one is proved.

\section{Model description and problem formulation}
\label{sec:model}

Let a stochastic basis $(\Omega, \mathcal{F}, \mathbb{F}=(\mathcal{F}_t)_{t \ge 0}, \mathbb{P})$ satisfying the usual conditions be given. We consider the capital dynamics $X = (X_t)_{t \ge 0}$ of an insurance company that makes annuity payments and invests its reserves in the financial market.

The evolution of the capital consists of operational and investment activities. The company makes continuous payouts at a constant rate $c > 0$. The stream of incoming payments (for example, the release of reserves upon contract termination) is modeled by a compound Poisson process $P_t = \sum_{i=1}^{N_t} \xi_i$, where $N$ is a Poisson process with intensity $\lambda > 0$, and $\{\xi_i\}_{i \ge 1}$ is a sequence of independent positive random variables with cumulative distribution function $F(x)$ and a finite first moment. To prove the uniqueness of the solution (Section \ref{sec:comparison}), it is assumed that the topological support of the measure $F(\diff x)$ is the set $(0, \infty)$, that is, $F(I) > 0$ for any open interval $I \subset (0, \infty)$.

The company's reserve is continuously invested in a risky asset whose price follows a geometric Brownian motion with expected return $a$ and volatility $\sigma > 0$. Thus, the capital dynamics is described by the stochastic differential equation:
\begin{equation*}
\diff X_t = (a X_t - c)\diff t + \sigma X_t \diff W_t + \diff P_t, \quad X_0 = u \ge 0,
\end{equation*}
where $W$ is a standard Brownian motion independent of the process $P$.

To ensure a positive expected capital growth for large values of $u$, we assume the net profit condition holds:
\begin{equation}
\label{eq:net_profit}
a > \frac{\sigma^2}{2} \quad \left(\text{equivalent to } \gamma := \frac{2a}{\sigma^2} > 1\right).
\end{equation}

The time of ruin is defined as the time the process first reaches the negative half-plane:
$\tau = \inf\{t \ge 0 : X_t \le 0\}$, with the convention $\inf \varnothing = \infty$.

Since the distribution of jumps is concentrated on the positive half-axis ($F(0)=0$), crossing zero by a jump is impossible. The capital drops to zero exclusively due to the continuous components (negative drift and diffusion). Consequently, on the event $\{\tau < \infty\}$, it holds that $X_\tau = 0$ almost surely.

The object of interest is the survival probability $\Phi(u) = \mathbb{P}(\tau = \infty \mid X_0 = u)$. For the non-local integral operators to function correctly, we extend the function by zero in the ruin domain: $\Phi(x) = 0$ for $x < 0$. 

Let us examine the asymptotic properties of the function $\Phi(u)$ at the boundaries of its domain.

\begin{lemma}
\label{lem:boundary_zero}
The following limit holds: $\lim_{u \to 0+} \Phi(u) = 0$. Setting $\Phi(0)=0$ yields a function that is right-continuous at zero.
\end{lemma}

\begin{proof}
Consider an auxiliary jump-free process $Y$ defined by the equation 
$$\diff Y_t = (a Y_t - c)\diff t + \sigma Y_t \diff W_t, \quad Y_0=u.$$
Since $c>0$, the boundary $0$ is accessible for the process $Y$ (see the classification of boundaries in \cite[Part II, Sec. 13]{Borodin2002}), and thus the hitting time of zero $T_u(Y) = \inf\{t \ge 0 : Y_t \le 0\}$ is almost surely finite. As $u \to 0+$, the exit time $T_u(Y)$ tends to zero in probability.

The survival event for the process $X$ is possible only if the first positive jump $\tau_1 \sim \text{Exp}(\lambda)$ occurs strictly before the diffusion part hits zero. Indeed, if $\tau_1 > T_u(Y)$, the process $X$ coincides with $Y$ on the entire interval $[0, T_u(Y)]$, making ruin inevitable. Bounding the probabilities, we obtain:
$$ \Phi(u) \le \mathbb{P}(\tau_1 \le T_u(Y)) = \mathbb{E}[1 - e^{-\lambda T_u(Y)}]. $$
Due to the boundedness and continuity of the function $t \mapsto 1 - e^{-\lambda t}$, the expectation tends to zero as $T_u(Y) \xrightarrow{\mathbb{P}} 0$.
\end{proof}

\begin{lemma}
\label{lem:boundary_inf}
Under condition \eqref{eq:net_profit}, the equality $\lim_{u \to \infty} \Phi(u) = 1$ holds.
\end{lemma}

\begin{proof}
Since $a > \sigma^2/2$, there exists a sufficiently small $\delta > 0$ such that $a - \delta > \sigma^2/2$. Consider a geometric Brownian motion $Z$ given by the equation $\diff Z_t = (a - \delta)Z_t \diff t + \sigma Z_t \diff W_t$ with the initial condition $Z_0 = u$. 

In the region of large capital values $X_t \ge c/\delta$, the drift of the process $X$ dominates the drift of the process $Z$: $a X_t - c \ge (a-\delta)X_t$. Since the jumps $P_t$ are positive and are added to the overall diffusion process, the classical comparison theorem for stochastic differential equations (see, for example, \cite[Chapter VI, Theorem 1.1]{IkedaWatanabe1981}) implies that $X_t \ge Z_t$ almost surely up to the moment $X_t$ first drops below the level $c/\delta$. 

The parameters of the process $Z_t$ satisfy condition \eqref{eq:net_profit}, hence $Z_t \to \infty$ a.s. as $t \to \infty$. The probability that the trajectory of $Z_t$ ever drops to the level $c/\delta$ tends to zero as $u \to \infty$. Consequently, the probability of the dominating process $X_t$ reaching this level (and, therefore, the probability of ruin) also tends to zero. Thus, $\lim_{u \to \infty} \Phi(u) = 1$.
\end{proof}

Thus, it is required to find a solution $\Phi(u)$ in the class of bounded functions satisfying the boundary conditions:
\begin{equation}
\label{eq:boundary_conditions}
\Phi(0) = 0, \quad \lim_{u \to \infty} \Phi(u) = 1.
\end{equation}

\section{Definition of a viscosity solution}
\label{sec:viscosity_definition}

Since the coefficient of the leading derivative degenerates at zero (due to the factor $u^2$), the classical definition of a solution to the boundary value problem up to the boundary is inapplicable. In this regard, we will use the framework of viscosity solutions, adapted for integro-differential operators.

Let us consider the linear operator $\mathcal{L}$ acting on functions $\varphi \in C^2((0, \infty)) \cap C_b([0, \infty))$:
\begin{equation*}
\mathcal{L}\varphi(u) = \frac{\sigma^2 u^2}{2} \varphi''(u) + (a u - c)\varphi'(u) + \lambda \int_0^\infty[\varphi(u+y) - \varphi(u)] \diff F(y).
\end{equation*}
This operator is degenerate elliptic at the boundary but preserves local uniform ellipticity strictly inside the domain $(0, \infty)$. The global boundedness of the test functions ($\varphi \in C_b$) is necessary to ensure the proper integrability of the non-local term at infinity.

For an arbitrary locally bounded function $w:[0, \infty) \to \mathbb{R}$, we introduce its semicontinuous envelopes: the upper semicontinuous envelope $w^*(x) = \limsup_{y \to x} w(y)$ and the lower semicontinuous envelope $w_*(x) = \liminf_{y \to x} w(y)$.

The specific nature of the operator $\mathcal{L}$ lies in the positivity of the jumps ($y > 0$). The value of the integral at a point $x_0$ depends on the function's behavior exclusively on the ray $[x_0, \infty)$. This allows us to use a relaxed extremum condition, requiring a global touch only to the right of the point under consideration.

\begin{definition}
\label{def:viscosity_sol}
A function $w:[0, \infty) \to \mathbb{R}$ is called:
\begin{enumerate}
    \item A viscosity subsolution of the equation $\mathcal{L} w = 0$ in $(0, \infty)$, if for any point $x_0 \in (0, \infty)$ and any test function $\varphi \in C^2((0, \infty)) \cap C_b([0, \infty))$ satisfying the touching conditions:
    \begin{itemize}
        \item $w^*(x_0) = \varphi(x_0)$,
        \item $w^*(x) \le \varphi(x)$ locally in some neighborhood of $x_0$ and globally for all $x \ge x_0$,
    \end{itemize}
    the inequality $\mathcal{L}\varphi(x_0) \ge 0$ holds.

    \item A viscosity supersolution of the equation $\mathcal{L} w = 0$ in $(0, \infty)$, if for any point $x_0 \in (0, \infty)$ and any test function $\varphi \in C^2((0, \infty)) \cap C_b([0, \infty))$ satisfying the touching conditions:
    \begin{itemize}
        \item $w_*(x_0) = \varphi(x_0)$,
        \item $w_*(x) \ge \varphi(x)$ locally in some neighborhood of $x_0$ and globally for all $x \ge x_0$,
    \end{itemize}
    the inequality $\mathcal{L}\varphi(x_0) \le 0$ holds.

    \item A viscosity solution, if it is simultaneously a viscosity subsolution and a viscosity supersolution.
\end{enumerate}
\end{definition}

\begin{remark}
Since it will be shown later that the sought survival probability $\Phi(u)$ is continuous on $[0, \infty)$, its semicontinuous envelopes trivially coincide: $w^* = w_* = w$.
\end{remark}

\begin{remark}
This definition is consistent with the classical one: if $w \in C^2((0, \infty)) \cap C_b([0, \infty))$ satisfies the equation $\mathcal{L}w = 0$ pointwise, it is automatically a viscosity solution (it suffices to choose $w$ itself as the test function $\varphi$). 
\end{remark}

Let us now define the notion of a solution for the considered boundary value problem.

\begin{definition}
\label{def:viscosity_bvp_sol}
A function $\Phi:[0, \infty) \to [0, 1]$ is called a viscosity solution of the boundary value problem if it is a viscosity solution of the equation $\mathcal{L} \Phi = 0$ in the domain $(0, \infty)$ in the sense of Definition \ref{def:viscosity_sol} and satisfies the boundary conditions \eqref{eq:boundary_conditions}.
\end{definition}

\section{Existence of a viscosity solution}
\label{sec:existence}

In this section, we prove that the survival probability $\Phi(u)$ is a viscosity solution of the equation $\mathcal{L} \Phi = 0$ in the sense of Definition \ref{def:viscosity_sol}. 

If the function $\Phi$ belonged a priori to the class $C^2((0, \infty)) \cap C_b([0, \infty))$, applying It\^{o}'s formula to the bounded martingale $M_t = \Phi(X_{t \wedge \tau}) = \mathbb{P}(\tau = \infty \mid \mathcal{F}_{t \wedge \tau})$ would immediately yield the identity $\mathcal{L}\Phi = 0$. The theory of viscosity solutions enables us to formalize this deduction without any assumptions on smoothness, relying exclusively on the dynamic programming principle (DPP).

\begin{theorem}
\label{thm:DPP}
For any initial capital $u > 0$ and any finite stopping time $\theta$, the following identity holds:
$$ \Phi(u) = \mathbb{E}_u[\Phi(X_{\theta \wedge \tau})]. $$
\end{theorem}

\begin{proof}
Using the tower property of conditional expectation, we write
$$\Phi(u) = \mathbb{E}_u[\mathbf{1}_{\{\tau = \infty\}}] = \mathbb{E}_u \big[ \mathbb{E}_u[\mathbf{1}_{\{\tau = \infty\}} \mid \mathcal{F}_{\theta \wedge \tau}] \big].$$
On the set $\{\theta \ge \tau\}$, ruin has already occurred. Due to the specifics of the model (the absence of negative jumps), ruin occurs continuously; therefore, $X_\tau = 0$ a.s., which implies $\Phi(X_{\theta \wedge \tau}) = \Phi(0) = 0$. This perfectly matches the value of the indicator $\mathbf{1}_{\{\tau = \infty\}} = 0$.
On the set $\{\theta < \tau\}$, the process is in the domain of strictly positive values. By the strong Markov property of $X$ and the measurability of the event $\{\tau=\infty\} \in \mathcal{F}_\infty$, the conditional expectation $\mathbb{E}_u[\mathbf{1}_{\{\tau = \infty\}} \mid \mathcal{F}_{\theta \wedge \tau}]$ is almost surely equal to $\Phi(X_\theta)$. Combining these cases completes the proof.
\end{proof}

\begin{theorem}
The function $\Phi(u)$ is a viscosity subsolution of the equation $\mathcal{L} \Phi = 0$.
\end{theorem}

\begin{proof}
Let $x_0 \in (0, \infty)$ and let $\varphi \in C^2((0, \infty)) \cap C_b([0, \infty))$ be a test function satisfying the touching conditions: $\Phi^*(x_0) = \varphi(x_0)$ and $\Phi^*(x) \le \varphi(x)$ locally in some neighborhood of $x_0$ and globally for all $x \ge x_0$.

Assume the contrary: $\mathcal{L}\varphi(x_0) < 0$. By the continuity of the function $\mathcal{L}\varphi$, there exist constants $\delta > 0$ and $\varepsilon > 0$ such that $\mathcal{L}\varphi(x) \le -\delta$ for all $x \in B_\varepsilon(x_0) = (x_0-\varepsilon, x_0+\varepsilon) \subset (0, \infty)$. 

By the definition of the upper semicontinuous envelope, there exists a sequence of initial points $x_n \to x_0$ such that $\Phi(x_n) \to \Phi^*(x_0) = \varphi(x_0)$. Consider the process with the initial condition $X_0 = x_n$. Let us denote the first exit time from the $\varepsilon$-neighborhood as $\tau_n = \inf\{t > 0 : X_t \notin B_\varepsilon(x_0)\}$, and set $\theta_n = \tau_n \wedge h$ for a fixed small $h > 0$.

Applying It\^{o}'s formula to $\varphi(X_t)$ on the stochastic interval $[0, \theta_n]$, we obtain:
$$ \mathbb{E}_{x_n}[\varphi(X_{\theta_n})] = \varphi(x_n) + \mathbb{E}_{x_n} \int_0^{\theta_n} \mathcal{L}\varphi(X_s) \diff s \le \varphi(x_n) - \delta \cdot \mathbb{E}_{x_n}[\theta_n]. $$

According to the DPP (Theorem \ref{thm:DPP}), $\Phi(x_n) = \mathbb{E}_{x_n}[\Phi(X_{\theta_n})]$. At the stopping time $\theta_n$, the process $X$ is either inside $B_\varepsilon(x_0)$ or has left it. Since the jumps of the process are strictly positive, the trajectory cannot jump downwards. Consequently, before exiting $B_\varepsilon(x_0)$, the process never drops below the level $x_0 - \varepsilon$, meaning that $X_{\theta_n} \ge x_0 - \varepsilon$ always holds. By the touching condition, the inequality $\Phi^* \le \varphi$ holds locally in $B_\varepsilon(x_0)$ and globally on the ray $[x_0, \infty)$. Thus, for a sufficiently small $\varepsilon$, it holds on the entire set of possible values of $X_{\theta_n}$, yielding $\Phi(X_{\theta_n}) \le \varphi(X_{\theta_n})$ almost surely.

Replacing $\varphi$ by $\Phi$ on the left-hand side of the inequality, we arrive at:
$$ \Phi(x_n) \le \varphi(x_n) - \delta \cdot \mathbb{E}_{x_n}[\theta_n]. $$
Let us pass to the limit as $n \to \infty$. The left side $\Phi(x_n)$ converges to $\varphi(x_0)$ by the choice of the sequence $x_n$. The right side $\varphi(x_n)$ also converges to $\varphi(x_0)$ due to the continuity of the test function. From the theorem on the continuous dependence of trajectories of stochastic differential equations on the initial condition, it follows that the expectations of the stopping times converge (see, for example, \cite[Chapter 5]{KaratzasShreve1991} for the diffusion part, since for small $h$ the probability of a jump occurring before $\theta_n$ tends to zero): $\mathbb{E}_{x_n}[\theta_n] \to \mathbb{E}_{x_0}[\theta_0]$. 

In the limit, we get:
$$ \varphi(x_0) \le \varphi(x_0) - \delta \cdot \mathbb{E}_{x_0}[\theta_0]. $$
Since the starting point $x_0$ lies strictly inside the interval $B_\varepsilon(x_0)$, we have $\mathbb{E}_{x_0}[\theta_0] > 0$, which leads to an impossible inequality and proves the theorem.
\end{proof}

\begin{theorem}
The function $\Phi(u)$ is a viscosity supersolution of the equation $\mathcal{L} \Phi = 0$.
\end{theorem}

\begin{proof}
The reasoning is analogous. Let $\varphi \in C^2((0, \infty)) \cap C_b([0, \infty))$ touch $\Phi_*$ from below at $x_0$, that is, $\Phi_*(x_0) = \varphi(x_0)$. If we assume the contrary, $\mathcal{L}\varphi(x_0) > 0$, then $\mathcal{L}\varphi(x) \ge \delta > 0$ in some $\varepsilon$-neighborhood of $x_0$. 

We choose a sequence $x_n \to x_0$ for which $\Phi(x_n) \to \Phi_*(x_0) = \varphi(x_0)$. It\^{o}'s formula gives $\mathbb{E}_{x_n}[\varphi(X_{\theta_n})] \ge \varphi(x_n) + \delta \cdot \mathbb{E}_{x_n}[\theta_n]$. Using the DPP and the touching condition from below $\Phi \ge \Phi_* \ge \varphi$, we obtain $\Phi(x_n) \ge \varphi(x_n) + \delta \cdot \mathbb{E}_{x_n}[\theta_n]$. Passing to the limit as $n \to \infty$ leads to the contradictory inequality $\varphi(x_0) \ge \varphi(x_0) + \delta \cdot \mathbb{E}_{x_0}[\theta_0]$.
\end{proof}

\begin{corollary}
\label{cor:visc_existence}
The survival probability $\Phi(u)$ is a viscosity solution of the boundary value problem in the sense of Definition \ref{def:viscosity_bvp_sol}.
\end{corollary}

\section{Comparison principle and uniqueness}
\label{sec:comparison}

In this section, we establish a comparison principle that guarantees the uniqueness of the viscosity solution. Let us rewrite the operator by separating the zero-order term:
\begin{equation*}
\mathcal{L}\Phi(u) = \frac{\sigma^2 u^2}{2} \Phi''(u) + (a u - c)\Phi'(u) - \lambda \Phi(u) + \lambda \int_0^\infty \Phi(u+y) \diff F(y).
\end{equation*}
The term $-\lambda \Phi(u)$ with the positive coefficient $\lambda > 0$ provides the operator with the necessary strict monotonicity property.

\begin{theorem}
\label{thm:comparison}
Let $U \in B([0, \infty))$ be an upper semicontinuous viscosity subsolution, and $V \in B([0, \infty))$ be a lower semicontinuous viscosity supersolution of the equation $\mathcal{L} \Phi = 0$ in the domain $(0, \infty)$. If the boundary conditions $U(0) \le V(0)$ and $\limsup_{u \to \infty} (U(u) - V(u)) \le 0$ are satisfied, then $U(u) \le V(u)$ for all $u \in[0, \infty)$.
\end{theorem}

\begin{proof}
Assume the contrary: the supremum of the difference is strictly positive, that is, $M := \sup_{u \ge 0} (U(u) - V(u)) > 0$. From the given boundary conditions, it follows that the supremum cannot be attained at zero or at infinity. Therefore, if a global maximum point exists, it must lie strictly inside the domain $(0, \infty)$.

Let us introduce the objective function with doubled variables and a penalization parameter $\alpha > 0$:
$$ \Psi_\alpha(x, y) = U(x) - V(y) - \frac{\alpha}{2}|x - y|^2. $$
Due to the global boundedness of $U$ and $V$, the function $\Psi_\alpha$ attains its global maximum at some point $(\hat{x}_\alpha, \hat{y}_\alpha) \in[0, \infty)^2$. As $\alpha \to \infty$, the classical properties hold: $\alpha |\hat{x}_\alpha - \hat{y}_\alpha|^2 \to 0$ and $\hat{x}_\alpha, \hat{y}_\alpha \to \hat{x}$, where $U(\hat{x}) - V(\hat{x}) = M$. It is crucially important that the limit point $\hat{x}$ strictly belongs to the interval $(0, \infty)$.

We apply the non-local version of Ishii's lemma (see \cite[Theorem 1]{Barles2008}) to the maximum point $(\hat{x}_\alpha, \hat{y}_\alpha)$. There exist real numbers $X, Y \in \mathbb{R}$ satisfying the matrix inequality
\begin{equation}
\label{eq:ishii_matrix}
\begin{pmatrix} X & 0 \\ 0 & -Y \end{pmatrix} \le 3\alpha \begin{pmatrix} 1 & -1 \\ -1 & 1 \end{pmatrix},
\end{equation}
and such that the following viscosity inequalities hold (denoting $p_\alpha = \alpha(\hat{x}_\alpha - \hat{y}_\alpha)$):
\begin{align*}
\frac{\sigma^2}{2} \hat{x}_\alpha^2 X + (a \hat{x}_\alpha - c)p_\alpha - \lambda U(\hat{x}_\alpha) + \lambda \int_0^\infty U(\hat{x}_\alpha+z) \diff F(z) &\ge 0, \\
\frac{\sigma^2}{2} \hat{y}_\alpha^2 Y + (a \hat{y}_\alpha - c)p_\alpha - \lambda V(\hat{y}_\alpha) + \lambda \int_0^\infty V(\hat{y}_\alpha+z) \diff F(z) &\le 0.
\end{align*}
Subtracting the second inequality from the first and grouping the terms, we get:
\begin{multline*}
\frac{\sigma^2}{2} (\hat{x}_\alpha^2 X - \hat{y}_\alpha^2 Y) + a(\hat{x}_\alpha - \hat{y}_\alpha)p_\alpha - \lambda (U(\hat{x}_\alpha) - V(\hat{y}_\alpha)) + \\
+ \lambda \int_0^\infty (U(\hat{x}_\alpha+z) - V(\hat{y}_\alpha+z)) \diff F(z) \ge 0.
\end{multline*}
Let us pass to the limit as $\alpha \to \infty$ and analyze the limits of each term:
\begin{itemize}
    \item Multiplying the matrix inequality \eqref{eq:ishii_matrix} from the left by the vector $(\hat{x}_\alpha, \hat{y}_\alpha)$ and from the right by its transpose, we obtain the estimate $\hat{x}_\alpha^2 X - \hat{y}_\alpha^2 Y \le 3\alpha(\hat{x}_\alpha - \hat{y}_\alpha)^2$. Since $\alpha(\hat{x}_\alpha - \hat{y}_\alpha)^2 \to 0$, the diffusion term tends to zero.
    \item The drift term $a (\hat{x}_\alpha - \hat{y}_\alpha)p_\alpha = a\alpha(\hat{x}_\alpha - \hat{y}_\alpha)^2$ also tends to zero.
    \item Taking into account that $U(\hat{x}) - V(\hat{x}) = M$ and exploiting the semicontinuity, the zero-order term and the integral yield in the limit:
    $$ -\lambda M + \lambda \int_0^\infty W(\hat{x}+z) \diff F(z) \ge 0, $$
    where $W = U - V$.
\end{itemize}
Canceling by $\lambda > 0$ and bringing $M = \int_0^\infty M \diff F(z)$ under the integral sign (since $F$ is a probability measure), we arrive at the relation:
$$ \int_0^\infty (W(\hat{x}+z) - M) \diff F(z) \ge 0. $$

Since $M$ is the global supremum of the function $W$, the integrand $W(\hat{x}+z) - M$ is non-positive for all $z$. The integral of a non-positive function can be non-negative only if the function equals zero $F$-almost everywhere. Thus, $W(\hat{x}+z) = M$ for almost all $z$ in the support of the measure $F$.

We now use the assumption regarding the jump measure: the topological support of $F$ is the set $(0, \infty)$. Since the function $W = U - V$ is upper semicontinuous (as the difference between an upper and a lower semicontinuous function), the set $E = \{y > 0 : W(y) < M\}$ is open. If the set $E$ intersected the ray $(\hat{x}, \infty) = \hat{x} + \text{supp}(F)$, the measure of this intersection $F(E - \hat{x})$ would be strictly positive. This would lead to a strictly negative value of the integral, which is impossible.

Consequently, the intersection is empty, and $W(y) \equiv M$ for all $y > \hat{x}$. However, the given boundary conditions at infinity imply that:
$$ \limsup_{y \to \infty} W(y) \le \limsup_{y \to \infty} U(y) - \liminf_{y \to \infty} V(y) \le 0. $$
Since $W(y) = M$ for $y > \hat{x}$, we obtain $M \le 0$, which contradicts the assumption $M > 0$. This contradiction proves that the assumption $M > 0$ was false, and $U(u) \le V(u)$ holds over the entire domain.
\end{proof}

\begin{corollary}
The viscosity solution of the considered boundary value problem is unique in the class of bounded functions $B([0, \infty))$.
\end{corollary}

\begin{proof}
Let $\Phi_1$ and $\Phi_2$ be two viscosity solutions of the considered boundary value problem. According to Definition \ref{def:viscosity_bvp_sol}, $\Phi_1$ is a viscosity subsolution, and $\Phi_2$ is a viscosity supersolution of the equation $\mathcal{L}\Phi = 0$. Both functions belong to the class of bounded functions $B([0, \infty))$ and satisfy the same boundary conditions. Applying Theorem \ref{thm:comparison}, we obtain $\Phi_1 \le \Phi_2$. Swapping the roles of the functions, we analogously deduce $\Phi_2 \le \Phi_1$. Therefore, $\Phi_1 \equiv \Phi_2$.
\end{proof}

\section{Smoothness of the solution and transition to the classical solution}
\label{sec:regularity}

In the previous sections, the existence and uniqueness of the viscosity solution $\Phi(u)$ were established. The goal of this section is to prove that the obtained function actually belongs to the class $C^2((0, \infty))$ and, consequently, is a classical solution of the IDE. This is a key step that qualitatively distinguishes our result from works limited exclusively to the viscosity framework \cite{BK2015}.

A fundamental fact for the smoothness analysis is that the degeneration of the diffusion coefficient occurs exclusively at the boundary $u=0$. Strictly inside the domain $(0, \infty)$, the linear operator preserves the property of local uniform ellipticity.

\begin{theorem}
\label{thm:smoothness}
The viscosity solution $\Phi(u)$ of the equation $\mathcal{L} \Phi = 0$ belongs to the class $C^2((0, \infty))$.
\end{theorem}

\begin{proof}
The proof is carried out by the method of successive regularity improvement and is divided into three steps.

\textbf{Step 1. Continuity of $\Phi$.} The survival probability $\Phi(u)$ is bounded on $[0, \infty)$. From what was proved in Section \ref{sec:existence}, its upper semicontinuous envelope $\Phi^*$ is a viscosity subsolution, and the lower semicontinuous envelope $\Phi_*$ is a viscosity supersolution of the boundary value problem. The boundary conditions for the semicontinuous envelopes are satisfied: $\Phi^*(0) = \Phi_*(0) = 0$ by Lemma \ref{lem:boundary_zero}, and $\lim_{u \to \infty} \Phi^*(u) = \lim_{u \to \infty} \Phi_*(u) = 1$ by Lemma \ref{lem:boundary_inf}. Applying the comparison principle (Theorem \ref{thm:comparison}) to them yields $\Phi^* \le \Phi_*$ on $[0, \infty)$. Since by definition $\Phi_* \le \Phi^*$ always holds, we conclude that $\Phi_* \equiv \Phi^*$. This means that the function $\Phi(u)$ is continuous over its entire domain.

\textbf{Step 2. Continuity of the integral term.} Let us rewrite the equation $\mathcal{L}\Phi = 0$ as a linear non-homogeneous ordinary differential equation (ODE):
\begin{equation}
\label{eq:linearized_ode}
A(u)\Phi''(u) + B(u)\Phi'(u) - \lambda \Phi(u) = f(u),
\end{equation}
where $A(u) := \frac{1}{2}\sigma^2 u^2$, $B(u) := a u - c$, and the integral term is moved to the right-hand side:
$$ f(u) := -\lambda \int_0^\infty \Phi(u+y) \diff F(y). $$
We will show that the function $f(u)$ is continuous on $(0, \infty)$. From Step 1 and the boundary properties, it follows that the function $\Phi$ is continuous on the entire half-axis $[0, \infty)$ and has a finite limit as $u \to \infty$. Any function that is continuous on a closed half-axis and has a finite limit at infinity is uniformly continuous.

Fix an arbitrary $\varepsilon > 0$. By the uniform continuity of $\Phi$, there exists a $\delta > 0$ such that $|\Phi(x) - \Phi(\tilde{x})| < \varepsilon$ for any $x, \tilde{x} \ge 0$ satisfying the condition $|x - \tilde{x}| < \delta$. Let $u, v \in (0, \infty)$ be such that $|u - v| < \delta$. Then for any $y \ge 0$, we have $|(u+y) - (v+y)| = |u-v| < \delta$. Let us estimate the increment of $f$:
$$ |f(u) - f(v)| \le \lambda \int_0^\infty |\Phi(u+y) - \Phi(v+y)| \diff F(y) \le \lambda \int_0^\infty \varepsilon \diff F(y) = \lambda\varepsilon. $$
Since the bound $\lambda\varepsilon$ does not depend on the choice of the points $u$ and $v$, the function $f$ is uniformly continuous, which implies $f \in C((0, \infty))$.

\textbf{Step 3. Upgrading smoothness to $C^2$.} Consider the equation \eqref{eq:linearized_ode} on an arbitrary compact set $K = [u_1, u_2] \subset (0, \infty)$. The coefficients $A(u)$ and $B(u)$ are infinitely differentiable, with $A(u) \ge \frac{1}{2}\sigma^2 u_1^2 > 0$ on $K$, and the right-hand side $f(u)$ is continuous, as proved in Step 2.

We consider the Dirichlet boundary value problem for the linear elliptic ODE \eqref{eq:linearized_ode} on the compact set $K$ with boundary conditions matching the values of the original function $\Phi$ at the endpoints $u_1$ and $u_2$. From the classical theory of differential equations, it follows that this local problem admits a unique classical solution $\tilde{\Phi} \in C^2(K)$.

Since $\tilde{\Phi}$ is a classical solution, it is also a viscosity solution on the compact set $K$. The original continuous function $\Phi$ is also a viscosity solution on $K$. Due to the local comparison principle for viscosity solutions of linear ODEs on a bounded interval (see, for example, \cite{Crandall1992}), these solutions are identically equal: $\Phi \equiv \tilde{\Phi}$ on $K$. Consequently, $\Phi \in C^2(K)$. By the arbitrariness of the choice of $K$, we conclude that $\Phi \in C^2((0, \infty))$.
\end{proof}

The following theorem is the main result of the paper.

\begin{theorem}
The survival probability $\Phi(u)$ is the unique classical solution of the boundary value problem:
\begin{equation*}
\begin{cases}
    \frac{\sigma^2 u^2}{2} \Phi''(u) + (a u - c)\Phi'(u) - \lambda \Phi(u) + \lambda \int_0^\infty \Phi(u+y) \diff F(y) = 0, & u > 0, \\
    \Phi(0) = 0, \\
    \lim_{u \to \infty} \Phi(u) = 1.
\end{cases}
\end{equation*}
By a classical solution, we understand a function $\Phi \in C^2((0, \infty)) \cap C_b([0, \infty))$ that satisfies the equation pointwise.
\end{theorem}

\begin{proof}
\textbf{Existence.} Corollary \ref{cor:visc_existence} states that $\Phi(u)$ is a viscosity solution of the considered boundary value problem. Theorem \ref{thm:smoothness} guarantees that $\Phi \in C^2((0, \infty))$. From the consistency property of the theory of viscosity solutions (see \cite{Crandall1992}), it follows that any twice continuously differentiable viscosity solution satisfies the differential equation in the classical (pointwise) sense. The fulfillment of the boundary conditions is justified in Lemmas \ref{lem:boundary_zero} and \ref{lem:boundary_inf}.

\textbf{Uniqueness.} Let $\tilde{\Phi} \in C^2((0, \infty)) \cap C_b([0, \infty))$ be another classical solution of the considered boundary value problem. Since $\tilde{\Phi}$ is twice continuously differentiable, it satisfies the equation pointwise and, by the fundamental consistency property, is automatically a viscosity solution. Because $\tilde{\Phi}$ belongs to the class of bounded functions $B([0, \infty))$ and satisfies the same boundary conditions, the comparison principle (Theorem \ref{thm:comparison}) can be applied to the functions $\Phi$ and $\tilde{\Phi}$. It immediately follows that $\tilde{\Phi} \equiv \Phi$.
\end{proof}

Thus, it has been proved that the survival probability is a sufficiently smooth function and satisfies the corresponding integro-differential equation in the classical sense.

\section*{Competing interests}
The author declares no competing interests.

\acknowledgement 
The author expresses his gratitude to Yuri Kabanov for the problem formulation.

\end{document}